
\def\LaTeX{%
  \let\Begin\begin
  \let\End\end
  \let\salta\relax
  \let\finqui\relax
  \let\futuro\relax}

\def\UK{\def\our{our}\let\sz s}
\def\USA{\def\our{or}\let\sz z}

\UK



\LaTeX

\USA


\salta

\documentclass[twoside,12pt]{article}
\setlength{\textheight}{24cm}
\setlength{\textwidth}{16cm}
\setlength{\oddsidemargin}{2mm}
\setlength{\evensidemargin}{2mm}
\setlength{\topmargin}{-15mm}
\parskip2mm


\usepackage[usenames,dvipsnames]{color}
\usepackage{amsmath}
\usepackage{amsthm}
\usepackage{amssymb,bbm}
\usepackage[mathcal]{euscript}
\usepackage{lmodern}
\usepackage{cite}
\usepackage{hyperref}
%
%
\usepackage[shortlabels]{enumitem}

\usepackage[ulem=normalem,draft]{changes}
%
%

%
 
\definecolor{viola}{rgb}{0.3,0,0.7}
\definecolor{ciclamino}{rgb}{0.5,0,0.5}
\definecolor{blu}{rgb}{0,0,0.7}
\definecolor{rosso}{rgb}{0.85,0,0}

\def\juerg #1{{\color{green}#1}}
\def\pier #1{{\color{red}#1}} 
\def\betti #1{{\color{blue}#1}}

\def\juerg #1{#1}
\def\pier #1{{#1}}
\def\betti #1{{#1}}




\bibliographystyle{plain}


%
\newtheorem{theorem}{Theorem}[section]

\finqui

\def\Beq{\Begin{equation}}
\def\Eeq{\End{equation}}

\def\Bthm{\Begin{theorem}}
\def\Ethm{\End{theorem}}
\def\Blem{\Begin{lemma}}
\def\Elem{\End{lemma}}
\def\Bprop{\Begin{proposition}}
\def\Eprop{\End{proposition}}
\def\Bcor{\Begin{corollary}}
\def\Ecor{\End{corollary}}
\def\Brem{\Begin{remark}\rm}
\def\Erem{\End{remark}}

\def\Bdim{\Begin{proof}}
\def\Edim{\End{proof}}
\def\Bcenter{\Begin{center}}
\def\Ecenter{\End{center}}
\let\non\nonumber




\def\step #1 \par{\medskip\noindent{\bf #1.}\quad}
\def\jstep #1: \par {\vspace{2mm}\noindent\underline{\sc #1 :}\par\nobreak\vspace{1mm}\noindent}

\def\Lip{Lip\-schitz}

\def\lhs{left-hand side}
\def\rhs{right-hand side}



\def\multibold #1{\def\arg{#1}%
  \ifx\arg\pto \let\next\relax
  \else
  \def\next{\expandafter
    \def\csname #1#1#1\endcsname{{\bf #1}}%
    \multibold}%
  \fi \next}

\def\pto{.}

\def\multical #1{\def\arg{#1}%
  \ifx\arg\pto \let\next\relax
  \else
  \def\next{\expandafter
    \def\csname cal#1\endcsname{{\cal #1}}%
    \multical}%
  \fi \next}


\def\multimathop #1 {\def\arg{#1}%
  \ifx\arg\pto \let\next\relax
  \else
  \def\next{\expandafter
    \def\csname #1\endcsname{\mathop{\rm #1}\nolimits}%
    \multimathop}%
  \fi \next}

\multibold
qwertyuiopasdfghjklzxcvbnmQWERTYUIOPASDFGHJKLZXCVBNM.

\multical
QWERTYUIOPASDFGHJKLZXCVBNM.

\multimathop
diag dist div dom mean meas sign supp .


\def\Accorpa #1#2 #3 {\gdef #1{\eqref{#2}--\eqref{#3}}%
  \wlog{}\wlog{\string #1 -> #2 - #3}\wlog{}}


\def\separa{\noalign{\allowbreak}}

\def\<#1>{\mathopen\langle #1\mathclose\rangle}
\def\norma #1{\mathopen \| #1\mathclose \|}

\def\I2 #1{\int_{Q_t}|{#1}|^2}
\def\IT2 #1{\int_{Q_t^T}|{#1}|^2}
\def\IO2 #1{\norma{{#1(t)}}^2}

\def\next{\\ & \quad}

\def\intQ{\int_Q}
\def\iO{\int_\Omega}

\def\Qtt{\int_{Q_t^T}}

\def\dtt{\partial_{tt}}
\def\dt{\partial_t}
\def\dn{\partial_{\bf n}}
\def\S{{\cal S}}

\def\X{{\cal X}}

\def\Uad{{\cal U}_{\rm ad}}
\def\G{{\cal G}}

\def\checkmmode #1{\relax\ifmmode\hbox{#1}\else{#1}\fi}


\def\erre{{\mathbb{R}}}
\def\rz{{\mathbb{R}}}

\def\enne{{\mathbb{N}}}

\def\J{{\cal J}}

\def\ph{\varphi}




\def\genspazio #1#2#3#4#5{#1^{#2}(#5,#4;#3)}
\def\spazio #1#2#3{\genspazio {#1}{#2}{#3}T0}

\def\L {\spazio L}
\def\H {\spazio H}
\def\W {\spazio W}

\def\C #1#2{C^{#1}([0,T];#2)}



\def\Lx #1{L^{#1}(\Omega)}
\def\Hx #1{H^{#1}(\Omega)}

\def\Ldue{\Lx 2}

\def\Huno{\Hx 1}
\def\Hdue{\Hx 2}

\def\Liq{{L^\infty(Q)}}
\def\Lio{{L^\infty(\Omega)}}
\def\Lzq{{L^2(Q)}}



\let\vp\varphi
\def\h{\mathbbm{h}}


\def\us{u^*}
\def\mus{\mu^*}
\def\phis{\vp^*}
\def\sis{\sigma^*} 
\def\xis{\xi^*}

\def\ual{u_\alpha}
\def\mual{\mu_\alpha}
\def\phial{\vp_\alpha}
\def\sial{\sigma_\alpha}

\def\mualn{\mu_{\alpha_n}}
\def\phialn{\vp_{\alpha_n}}

\def\pal{p_\alpha}
\def\qal{q_\alpha}
\def\ral{r_\alpha}
\def\paln{p_{\alpha_n}}
\def\qaln{q_{\alpha_n}}
\def\raln{r_{\alpha_n}}
\def\usal{u^*_\alpha}

\let\hat\widehat

\def\uad{{\cal U}_{\rm ad}}

\usepackage{amsmath}
\DeclareFontFamily{U}{mathc}{}
\DeclareFontShape{U}{mathc}{m}{it}%
{<->s*[1.03] mathc10}{}

\DeclareMathAlphabet{\mathscr}{U}{mathc}{m}{it}
\Begin{document}


%
\title{Asymptotic analysis of an optimal control problem for a phase field tumor growth model with 
hyperbolic relaxation of the chemical potential}

\author{}
\date{}
\maketitle
\Bcenter
\vskip-1.5cm
{\large\sc Pierluigi Colli$^{(1)}$}\\
{\normalsize e-mail: {\tt pierluigi.colli@unipv.it}}\\[0.25cm]
{\large\sc Elisabetta Rocca$^{(1)}$}\\
{\normalsize e-mail: {\tt elisabetta.rocca@unipv.it}}\\[0.25cm]
{\large\sc J\"urgen Sprekels$^{(2)}$}\\
{\normalsize e-mail: {\tt juergen.sprekels@wias-berlin.de}}\\[.5cm]
$^{(1)}$
{{\small Dipartimento di Matematica ``F. Casorati''}\\
{\small Universit\`a di Pavia}\\
{\small via Ferrata 5, I-27100 Pavia, Italy}\\[.3cm] 
$^{(2)}$
{\small Weierstrass Institute for Applied Analysis and Stochastics}\\
{\small Anton-Wilhelm-Amo-Strasse 39, D-10117 Berlin, Germany}\\[10mm]}
\Ecenter

\Begin{abstract}
\noindent 
In this paper, we study the optimal control of a phase field tumor growth model of Cahn--Hilliard type with possibly singular potential in which  the often assumed parabolic relaxation of the chemical potential is replaced  by a hyperbolic one.  We investigate the asymptotic behavior 
of the control problem as the relaxation parameter $\,\alpha\,$ approaches zero. It is shown that the corresponding state
and adjoint state variables of the relaxed problem converge in a suitable sense to their 
counterparts for the unrelaxed regime where $\alpha=0$.  Also, weak limits of optimal controls for the relaxed problem can be shown to be optimal controls for the unrelaxed case. Moreover, every fixed 
optimal control $\,\us\,$ for the unrelaxed system turns out to be the limit of a family 
$\{u^*_\alpha\}_{\alpha>0}$ of minimizers 
of suitable  optimal control problems for the relaxed systems, where the original cost functional
is replaced by the so-called ``adapted'' cost functional.  It is also shown that the 
variational inequalities satisfied by $\usal$, which constitute
the first-order necessary optimality conditions of the
relaxed problems with adapted cost functional, converge to the corresponding variational 
inequality fulfilled by $\us$ for the unrelaxed problem with the original (nonadapted) cost functional.


\End{abstract}

\salta
\pagestyle{myheadings}
\newcommand\testopari{\sc Colli -- Rocca -- Sprekels}
\newcommand\testodispari{\sc Asymptotics in the control of a relaxed tumor growth model}
\markboth{\testopari}{\testodispari}
\finqui
%


\section{Introduction}
\label{INTRO}
\setcounter{equation}{0}

Let  $\Omega\subset\erre^3$ denote some open and bounded domain having a smooth boundary $\Gamma=\partial\Omega$ with outward normal $\,{\bf n}\,$ and \pier{the} corresponding outward normal derivative $\dn$.  Moreover, we fix some final time $T>0$ and
introduce for every $t\in (0,T]$ the sets $Q_t:=\Omega\times (0,t)$ 
 and $\Sigma_t:=\Gamma\times (0,t)$,
 where we put, for the sake of brevity, $Q:=Q_T$ and $\Sigma:=\Sigma_T$.
We then consider for parameters $\,\alpha\ge 0$\, the following optimal control problem:

\vspace{3mm}\noindent
{\bf (CP$_\alpha$)} \quad Minimize the {\it cost functional}
\begin{align} 
\non	
\J(\vp,u)
:=&\,  \frac{b_1}2 \intQ |\vp-\widehat \vp_Q|^2
 + \frac{b_2}2 \iO |\vp(T)-\widehat\vp_\Omega|^2
+ \frac{b_3}2 \intQ |u|^2 
 \,+\,\kappa\,\G(u)
 \\[1mm]
 \label{cost} 
  =:& \,J(\vp,u) + \kappa \G(u)
\end{align}
subject to the state system
\begin{align}
\label{ss1}
&\alpha\dtt\mu+\dt\ph-\Delta\mu=P(\ph)(\sigma+\chi(1-\ph)-\mu) - \h(\ph)z
 \quad&&\mbox{in }\,Q\,,\\[1mm]
\label{ss2}
&\tau\dt\vp-\Delta\vp+\xi+F_2'(\vp)=\mu+\chi\,\sigma,
\quad \xi\in\partial F_1(\vp), \quad&&\mbox{in }\,Q\,,\\[1mm]
\label{ss3}
&\dt\sigma-\Delta\sigma=-\chi\Delta\vp-P(\ph)(\sigma+\chi(1-\ph)-\mu)+u\quad&&\mbox{in }\,Q\,,\\[1mm]
\label{ss4}
&\dn \mu=\dn\vp=\dn\sigma=0 \quad&&\mbox{on }\,\Sigma\,,\\[1mm] 
\label{ss5}
&\mu(0)=\mu^0,\quad \dt\mu(0)=\mu^1,\quad\vp(0)=\vp^0,\quad \sigma(0)=\sigma^0\,,\quad &&\mbox{in }\,\Omega\,,
\end{align}
\Accorpa\Statesys {ss1} {ss5}
and to the control constraint 
\begin{equation}
\label{uad}
u\in\uad:=\{v\in\Liq: \underline u\le v\le\overline u\,\,\,\mbox{a.e. in }\,Q\},
\end{equation}
where the threshold functions satisfy
\Beq
\label{thresholds}
\underline u,\overline u\in\Liq \,\,\,\mbox{and}\,\,\, \underline u\le\overline u\,\,\,
\mbox{a.e. in }\,Q.
\Eeq
We assume that $b_1\ge 0$, $b_2\ge 0$, $b_3>0$ and $\kappa\ge 0$ are prescribed
constants, $\widehat \vp_Q\in \Lzq$ and $\widehat\vp_\Omega\in \Huno$ are given 
target functions, and $\G:L^2(Q)\to\erre$ is a nonnegative, continuous and convex 
functional. A prototypical form of $\G$, which enhances the occurrence of
{\em sparsity}, is given by
\Beq
\label{defG}
\G(u)=\int_Q |u|=\|u\|_{L^1(Q)}.
\Eeq

The  state system \Statesys\ constitutes a simplified and relaxed version of the four-species thermodynamically consistent model for tumor growth
originally proposed by Hawkins-Daruud et al.\ in \cite{HZO} that additionally includes 
\juerg{chemotactic} terms. Notice  that  the first two initial conditions in \eqref{ss5} are 
 not needed if $\alpha=0$; in 
the following, we always tacitly assume they are present only if $\alpha>0$.

Let us briefly review the role of the symbols occurring in \Statesys. The primary variables $\mu,\ph, \sigma$ denote the phase field, the associated chemical potential, and the nutrient concentration, respectively.
Furthermore, we stress that the additional term $\alpha\dtt\mu$ 
is a hyperbolic \pier{relaxation} of equation \eqref{ss1},
whereas the term $\tau\dt\ph$ is the viscosity contribution to the Cahn--Hilliard equation. \betti{Notice that the parameter $\tau$ 
is strictly positive in the present contribution, while the aim is to let the hyperbolic relaxation parameter $\alpha$ tend  to zero.} 
The nonlinearity $P$ denotes a proliferation function, whereas the positive constant $\chi$ represents the chemotactic sensitivity.
 
The key idea behind the regularizations originates from the fact that 
their presence allows us to take into account more general potentials $F=F_1+F_2$
whose derivatives, which play the role of a thermodynamic driving force in the model, may be singular. Moreover, as a common feature of phase field models, $F$ is a nonlinearity that is assumed 
to possess a double-well shape. We consider here only potentials of {\em regular} and 
{\em logarithmic} type, whose prototypes are given, in this order, by
\begin{align}
\label{regpot}
&F_{\rm reg}(r)=\frac 14 \left(1-r^2\right)^2 \quad\mbox{for $\,r\in\rz$,}\\[1mm]
\label{logpot}
&F_{\rm log}(r)=\left\{\begin{array}{ll}
(1+r)\,\ln(1+r)+(1-r)\,\ln(1-r)-k_1 r^2\quad&\mbox{for $\,r\in (-1,1)$}\\
2\ln(2)-k_1\quad&\mbox{for $\,r\in\{-1,1\}$\,,}\\
+\infty\quad&\mbox{for $\,r\not\in[-1,1]$}
\end{array}\right.
\end{align}
where $k_1>1$ so that $F_{\rm log}$ is nonconvex. 
Observe that $F_{\rm log}$ is very relevant in the 
applications, where $F'_{\rm log}(r)$ becomes unbounded as $r\to\pm 1$.

In the above model equations, there are two functions that may 
serve as distributed controls acting in the phase and nutrient equations,
respectively. The  variable $z$, which is nonlinearly coupled to the state variable $\varphi$ in the phase equation \eqref{ss1}, models the application of a cytotoxic drug into 
the system; it is multiplied by a truncation function $\h$ in order to have the action
only in the spatial region where the tumor cells are located. We refer to \cite{GLSS, GARL_1, HKNZ, KL} for  possible choices of $\h$.
On the other hand, the control $u$ appearing in \eqref{ss3} can model  the supply of
either an external medication or some nutrient.

In this paper, we only consider $u$ as a control while $z$ is assumed fixed. The 
reason for that is purely technical: on the one hand, we intend to include
the logarithmic potential in our analysis. This makes it necessary to prove a so-called
{\em uniform separation condition} for the phase variable $\ph$, which is known
to be available (cf. 
\cite[Thm.~3.1 and Rem.~3.2]{CRS1})  only if $z\in L^2(0,T;\Huno)$. On the other hand, an analysis of the asymptotic behavior as $\alpha\searrow0$ seems only to be possible
(see \cite[Sect.~4]{CRS1}) if $z\in H^1(0,T;\Ldue)$ and $P$ is a positive constant function,
 in addition. However, from the 
viewpoint of optimal control it is very disadvantageous to postulate that the
admissible controls belong to $H^1(0,T;\Ldue)\cap L^2(0,T;\Huno)$. We have
therefore chosen to consider $z$ as a fixed datum and, moreover, to assume that $P$ is a
positive constant.  

\pier{\juerg{We now} provide an overview of the related literature. Diffuse interface models of Cahn--Hilliard type \juerg{form} a well-established framework for the description of tumor growth phenomena, \juerg{since they can both} capture the evolution of complex interfaces and the interaction between different species such as tumor cells, nutrients, and chemical agents. Starting from the multiphase formulation \juerg{proposed} by Hawkins-Daruud et al.~\cite{HZO}, various extensions and analytical studies of such \juerg{models} have been carried \juerg{out}. \juerg{The} contributions \cite{GLSS, GARL_1, GARLR,HKNZ, KL, S_b, SigTime} \juerg{discuss} different modeling aspects \juerg{such as} the choice of nonlinear couplings, chemotactic mechanisms, and the role of additional source terms. \juerg{In particular, physically relevant choices of the truncation function $\h$ and the proliferation function $P$ are given}.}

\pier{From the analytical viewpoint, Cahn--Hilliard systems with possibly singular free energies have been intensively studied (see, in particular\betti{\cite{AS, CGSS, CS, EG, FGR, FLRS, GHW, LRS, RSS}}). \juerg{Potentials of logarithmic type (such as \eqref{logpot}) are} particularly relevant from the modeling perspective, but \juerg{entail} mathematical difficulties due to the blow-up of the derivative at the boundary of the physical interval. A key tool in this context is the so-called uniform separation property, which guarantees that the phase variable remains strictly inside the domain of definition of the potential. For results that are related to the present setting, including well-posedness, regularity, and separation properties, we refer in particular to \cite{FGR, CGH, FLRS, GARL_1, CRS1}.}

\pier{Optimal control problems for phase field systems of Cahn--Hilliard type have also been widely investigated in recent years\juerg{. In this respect, we refer} to~\cite{CGS24, CS-control, CS-nonvisc,CST} for viscous or nonviscous Cahn--Hilliard systems, \cite{ACGW, FLRS, GLSS, SW} for Cahn--Hilliard--Darcy systems, \cite{AS,CGSS,GSS} for Cahn--Hilliard--Keller--Segel systems, \betti{\cite{CRW, CGRS3, CSS1, CSS2, CSSJOTA,EK_ADV,S_DQ,S_a, GHW, LRS, RSS, ST, ST24,You}} for tumor growth models, \cite{DRST, Funo, Fdue} for nonlocal models. Different control mechanisms have been considered, including distributed controls acting in the bulk \juerg{and} boundary controls, as well as cost functionals involving tracking-type terms and sparsity-promoting contributions. However, most of the available literature focuses on parabolic models. In contrast, the incorporation of relaxation effects, and in particular hyperbolic relaxation terms, has received comparatively little attention in the context of tumor growth models and their optimal control. The introduction of a hyperbolic relaxation of the chemical potential, as in equation~\eqref{ss1}, is motivated by the desire to account for inertial effects and to provide a more accurate description of transient dynamics (see \cite{CS, CS-control, CRS1}). At the same time, \juerg{the presence of such a relaxation term is  interesting from the viewpoint of mathematical analysis, since} it yields additional time regularity for the chemical potential.}

\pier{\juerg{It is therefore natural to investigate the asymptotics of the system as  the relaxation parameter $\alpha$ approaches zero.} For the underlying state system, this problem has been investigated in \cite{CRS1, CRS2}, where convergence results have been established under suitable assumptions on the data. Nevertheless, the corresponding asymptotic analysis for optimal control problems is much more delicate, since it involves not only the convergence of the state variables, but also the behavior of optimal controls and adjoint variables, as well as the passage to the limit in the associated optimality conditions. The main goal of the present paper is to address these issues. We study an optimal control problem for a tumor growth model with hyperbolic relaxation of the chemical potential and investigate the asymptotic behavior as \juerg{$\alpha \searrow 0$}. In particular, we prove that optimal controls for the relaxed problems converge (in a suitable sense) to optimal controls for the limit problem, and we show that every optimal control of the limit problem can be approximated by a sequence of optimal controls of suitably adapted problems. Moreover, we analyze the convergence of the associated variational inequalities, which represent the first-order necessary optimality conditions.}

\pier{The paper is organized as follows. In Section~\ref{STATE}, we introduce the functional setting, state the assumptions, and recall well-posedness results for the state system. We also derive uniform estimates and discuss the convergence of solutions as $\alpha \searrow 0$. In Section~\ref{OPT}, we investigate the optimal control problems, establishing existence results and studying the asymptotic behavior of optimal controls. In particular, we introduce adapted cost functionals in order to obtain strong convergence results and to approximate arbitrary optimal controls of the limit problem. Sections~\ref{ADJ} and~\ref{CONV} are devoted to the derivation of first-order necessary optimality conditions for the relaxed problems and to the analysis of their convergence to the corresponding conditions for the limit problem.}

Throughout the paper, we \pier{repeatedly use} H\"older's inequality, the elementary Young's inequality
\begin{equation}
\label{Young}
a b\,\le \delta |a|^2+\frac 1{4\delta}|b|^2\quad\forall\,a,b\in\erre, \quad\forall\,\delta>0,
\end{equation}
as well as the continuity of the embeddings $H^1(\Omega)\subset L^p(\Omega)$ for $1\le p\le 6$ and 
$\Hdue\subset C^0(\overline\Omega)$. Notice that the latter embedding is also compact, while this holds true
for the former embeddings only if $p<6$. 

Finally, let us introduce a convention that will be tacitly employed throughout the paper:
the symbol  $C$ is used to indicate every constant
that depends only on the structural data of the problem (such as
$T$, $\Omega$, $\chi$,  $\tau$, the shape of the 
nonlinearities, and the norms of the involved functions),
so that its meaning may change from line to line and even within formulas.
If a parameter $\delta$ enters the computation, then
the symbol $C_\delta$ denotes constants
that additionally depend on~$\delta$.
On the contrary, precise constants that we refer to
will be denoted in a different way.

\section{General setting and convergence of the states}
\label{STATE}
\setcounter{equation}{0}

In this section, we introduce the general setting of our 
problem, state a well-posedness results for the state system \eqref{ss1}--\eqref{ss5}, and
discuss the convergence properties of the state variables as $\alpha\searrow0$. 
To begin with, for a Banach space $\,X\,$ we denote by $\|\cdot\|_X$
the norm in the space $X$ or in a power thereof, and by $\,X^*\,$ its dual space. 
The only exception from this rule applies to the norms of the
$\,L^p\,$ spaces and of their powers, which we often denote by $\|\cdot\|_p$, for
$\,1\le p\le +\infty$. As usual, for Banach spaces $\,X\,$ and $\,Y\,$ that are contained in the same topological vector space, we introduce the linear space
$\,X\cap Y\,$ which becomes a Banach space when endowed with its natural norm $\,\|u\|_{X \cap Y}:=
\|u\|_X\,+\,\|u\|_Y$, for $\,u\in X\cap Y$.
Moreover, we introduce the spaces
\begin{align}
  & H := \Ldue \,, \quad  
  V := \Huno\,,   \quad
  W := \{v\in\Hdue: \ \dn v=0 \,\mbox{ on $\,\Gamma$}\}.
  \label{defHVW}
\end{align}
Furthermore, by $\,(\,\cdot\,,\,\cdot\,)$, $\,\Vert\,\cdot\,\Vert$, and $\<\cdot,\cdot>$, we denote the standard inner product 
and related norm in $\,H$, as well as the dual product between $V$ and its dual $V^*$. We then have the dense and compact embeddings $\,V\subset H\subset  V^*$, with the standard identification  \,\,$\langle v,w \rangle =(v,w)\,\,$ for
 $v\in H$ and $w\in V$.  
 
 We now provide precise assumptions for the data of the system \Statesys. 
\begin{enumerate}[label={\bf (A\arabic{*})}, ref={\bf (A\arabic{*})}]
\item \label{const:weak}
	$\tau $ and $\chi$ are positive constants.
\item \label{F:weak}
	$F=F_1+F_2$ satisfies: $\,F_2\in C^3(\erre)\,$ has a Lipschitz continuous derivative $\,F_2^{\,\prime}$; 
$\,F_1:\erre\to [0,+\infty]\,$ is convex and lower semicontinuous with $\,F_1(0)=0$, and there are constants $\,-\infty\le r_-<0<r_+\le +\infty\,$ such that the restriction of $F_1$ to $(r_-,r_+)$ belongs to $C^3(r_-,r_+)$  and satisfies   
\Beq
\label{singular}
\lim_{r\searrow r_-} F_1^{\,\prime}(r)=-\infty \qquad\mbox{and}\qquad \lim_{r\nearrow r_+} F_1^{\,\prime}(r)=+\infty\,.
\Eeq
\item \label{P:weak}
	$P\equiv P_0>0$ is a constant function, and $z\in H^1(0,T;H)
	\cap L^2(0,T;V)\cap\Liq$.
\item \label{h:weak}
	$\h \in C^2(\erre)$ is nonnegative, bounded, and $\h$ and $\h'$ are \Lip\ continuous.
\end{enumerate}
Let us note that both the potentials  \eqref{regpot} and \eqref{logpot} fulfill the condition 
$\ref{F:weak}$, where in the latter case we have $(r_-,r_+)=(-1,1)$. We also observe that the
subdifferential $\partial F_1$ of the convex function $F_1$ is a maximal monotone graph
in $\erre\times\erre$ satisfying $0\in \partial F_1(0)$. Its effective domain is by virtue of the blow-up condition \eqref{singular} given by  $D(\partial F_1)=(r_-,r_+)$, and it holds
 $\partial F_1(r)=\{F_1'(r)\}$  for every 
$r\in(r_-,r_+)$.  Moreover, the subdifferential operators induced by $F_1$ 
in $L^2$ spaces, which are again denoted by $\partial F_1$, are maximal
monotone in these spaces.

By a solution to  the state system \Statesys\  we mean a 
triple $(\mu,\ph,\sigma)$ satisfying 
\begin{align}
\label{mutglatt}
&\alpha\mu\in H^2(0,T;V^*)\cap W^{1,\infty}(0,T;H),\\
\label{muglatt}
&\mu\in L^\infty(0,T;V),\\
\label{phiglatt}
&\ph\in W^{1,\infty}(0,T;H)\cap H^1(0,T;V)\cap L^\infty(0,T;W)\cap C^0(\overline Q),\\
\label{sigmaglatt}
&\sigma\in H^1(0,T;H)\cap C^0([0,T];V) \cap L^2(0,T;W),\\
\label{xiglatt}
&\xi\in L^\infty(0,T;H),
\end{align}
that solves the following system:
\begin{align}
 \label{var:1}
 & \<\alpha\dtt\mu , v > + \iO\dt\ph \, v 
	+ \iO \nabla \mu \cdot \nabla v
	= \iO P_0(\sigma+\chi(1-\ph)-\mu)v
	-\iO \h(\ph)z v
	\nonumber\\
	& \qquad \hbox{for every $v \in V $ and a.e. in $(0,T)$,}\\[2mm]
\separa	
\label{var:2}
	&\tau\dt\vp-\Delta\vp+\xi+F_2'(\vp)=
	\mu+\chi\,\sigma, \quad \xi\in\partial F_1(\ph), \quad \hbox{a.e. in $\,Q$,}
 	\\
 \separa
\label{var:3}
 	& \dt\sigma -\Delta\sigma
 	=-\chi\Delta\ph 
 	-  P_0(\sigma+\chi(1-\ph)-\mu)
 	+ u	 \, \,\,\mbox{ a.e. in \,$Q$,}\\
\label{var:4}
	&\mu(0)=\mu^0, \quad
\dt\mu(0)=\mu^1,\quad
	\ph(0)=\ph^0, \quad
	\sigma(0)=\sigma^0, \quad \hbox{a.e. in $\,\Omega$}.
\end{align}
Again, it is understood that the first two initial conditions in \eqref{var:4}  are not
present if $\alpha=0$. Observe also that the homogeneous Neumann boundary 
condition for $\mu,\ph,\sigma$ are encoded in \eqref{var:1} and the
regularity requirements
\eqref{phiglatt}, \eqref{sigmaglatt}, respectively. The following 
well-posedness result for positive $\alpha$ is a \pier{consequence of \cite[Theorems~2.2 and~3.1]{CRS1}.}
\Bthm
Let \ref{const:weak}--\ref{h:weak} 
be fulfilled, and assume that 
\begin{align}
\label{initialreg}
&\mu^0\in W,\quad \mu^1\in V,\quad \ph^0\in W,\quad \sigma^0\in \pier{V\cap \Lio},\\
\label{initialsep}
&r_-<\ph^0(x)<r_+\quad\mbox{for all \,$x\in\overline\Omega$}. 
\end{align}
Then the state system \Statesys\ has for every $u\in L^\infty(0,T;H)$ and every 
$\alpha\in (0,1]$ a unique solution
$(\mual,\phial,\sial)$ such that
\begin{align}
\label{regmu}
&\mual\in W^{2,\infty}(0,T;V^*)\cap W^{1,\infty}(0,T;H)\cap H^1(0,T;V)\cap L^\infty(0,T;W)
\cap C^0(\overline Q),\\
\label{regphi}
&\phial\in  W^{1,\infty}(0,T;H)\cap H^1(0,T;V)\cap L^\infty(0,T;W)\cap C^0(\overline Q),\\
\label{regsigma}
&\sial \in H^1(0,T;H)\cap C^0([0,T];V)\cap L^\infty(0,T;W)\cap \Liq,\\
\label{regf1}
&F_1'(\phial)\in\Liq.
\end{align}
Moreover, there are constants $K_1^\alpha>0$ and $r_-<
\bar\rho_\alpha<\hat\rho_\alpha<r_*$, which depend only 
on the data of the state system and the norm $\,\|u\|_{L^\infty(0,T;H)}$, such that
\begin{align}
\label{ssbound1}
&\|\mual\|_{ W^{2,\infty}(0,T;V^*)\cap W^{1,\infty}(0,T;H)\cap H^1(0,T;V)
\cap L^\infty(0,T;W)
\cap C^0(\overline Q)} \nonumber\\
&+\|\phial\|_{W^{1,\infty}(0,T;H)\cap H^1(0,T;V)\cap L^\infty(0,T;W)\cap C^0(\overline Q)}
\nonumber\\
&+\|\sial\|_{H^1(0,T;H)\cap C^0([0,T];V)\cap L^\infty(0,T;W)\cap \Liq}\,\le\,K_1^\alpha\,,
\\[2mm]
\label{separaal}
&\bar\rho_\alpha \le \phial(x,t) \le \hat\rho_\alpha\quad\mbox{for all }\,(x,t)\in \overline Q\,.
\end{align} 
\Ethm
\Brem
From Theorem 2.1 and the boundedness of $\uad$ in $L^\infty(0,T;H)$ it follows that
the same constants $K_1^\alpha>0$  and $\bar\rho_\alpha,\hat\rho_\alpha$ can be chosen for all admissible 
controls $u\in\uad$.  In particular, it follows that for every fixed $\alpha\in (0,1]$ the uniform separation property for the phase variable
is satisfied.  This implies that $\partial F_1(\phial)=\{F_1'(\phial)\}$ a.e. in $Q$ and that the phase equation \eqref{var:2} is fulfilled in the strong form
\Beq
\label{var:2neu}
\tau\dt\phial-\Delta\phial+F'(\phial)=\mual+\chi\sial \quad\mbox{a.e. in }\,Q\,.
\Eeq
Moreover, the control-to-state
operator
\Beq
\label{defSal}
\S_\alpha: L^\infty(0,T;H)\ni u \mapsto \S_\alpha(u)=(\S_\alpha^1(u),\S_\alpha^2(u),
\S_\alpha^3(u)):=(\mual,\phial,\sial)
\Eeq
is well defined as a mapping from $L^\infty(0,T;H)$ into the solution space
\begin{align}
\label{deX}
\X\,:=\,&\bigl(W^{2,\infty}(0,T;V^*)\cap W^{1,\infty}(0,T;H)\cap H^1(0,T;V)
\cap L^\infty(0,T;W)
\cap C^0(\overline Q) \bigr) \nonumber\\
&\times\bigl(W^{1,\infty}(0,T;H)\cap H^1(0,T;V)\cap L^\infty(0,T;W)\cap C^0(\overline Q)
\bigr)\nonumber\\
&\times\bigl(H^1(0,T;H)\cap C^0([0,T];V)\cap L^\infty(0,T;W)\cap \Liq\bigr)\,,
\end{align}
and the optimal control problem {\bf (CP$_\alpha$)} is well defined for
$\alpha\in (0,1]$.
\Erem
Next, we turn our \pier{attention} to the limiting case $\alpha=0$ in which the first term in
\eqref{var:1} and the first two initial conditions in \eqref{var:4} are not present. As \pier{it} has
been shown in\pier{\cite{CRS1} (see, in particular, \cite[Remarks~4.2 and~4.4]{CRS1})}, there exists a unique quadruple $(\mu_0,\ph_0,
\sigma_0,\xi_0)$ with the regularity
\begin{align}
\label{regmuz}
&\mu_0\in L^\infty(0,T;W),\\
\label{regphiz}
&\ph_0\in W^{1,\infty}(0,T;H)\cap H^1(0,T;V)\cap L^\infty(0,T;W)\cap C^0(\overline Q),\\
\label{regsigmaz}
&\sigma_0\in H^1(0,T;H)\cap C^0([0,T];V)\cap L^2(0,T;W)\cap \Liq,\\
\label{xiz}
&\xi_0\in \Liq,
\end{align}
which satisfies the system \eqref{var:1}--\eqref{var:4}. 
\pier{We emphasize that our notation for the solution should not be confused with that used for the initial data (cf.~\eqref{ss5} and~\eqref{var:4}), where the symbol $0$ appears as a superscript.}
Notice that the condition $\xi_0
\in\partial F_1(\ph_0)$ implies that $r_-<\ph_0(x,t)<r_+$ and thus $\xi_0(x,t)=F_1'(\ph_0
(x,t))$ for a.e. $(x,t)\in Q$. Therefore, the phase equation \eqref{var:2} is, also in the
case $\alpha=0$, satisfied in the strong form \eqref{var:2neu}.  
In addition, from
\eqref{regmuz}--\eqref{regsigmaz} and \ref{F:weak} it follows that  $-F_2'(\ph_0)+
\mu_0+\chi\sigma_0\in\Liq$. A standard argument using the blow-up condition
\eqref{singular} then shows that there exist constants $r_-<\bar\rho_0<\hat\rho_0<r_+$ 
such that
\begin{equation}
\label{separao}
\bar\rho_0\le \ph_0(x,t)\le \hat\rho_0  \quad\mbox{for all }\,(x,t)\in\overline Q.
\end{equation} 
That is, also in the case $\alpha=0$ a uniform separation condition is satisfied. 

Moreover, we may introduce the control-to-state operator
\Beq
\label{defSz}
\S_0:L^\infty(0,T;H)\ni u\mapsto \S_0(u)=(\S_0^1(u),\S_0^2(u),\S_0^3(u),\S_0^4(u))
:=(\mu_0, \ph_0,\sigma_0,\xi_0 )\,,
\Eeq
and also the control problem {\bf (CP$_0$)} is well defined. \pier{Of 
course, in this case the solution space differs in the first component and 
and includes an additional fourth component.}

We are now going to prove the following approximation result.
\Bthm
Let the assumptions \ref{const:weak}--\ref{h:weak}, \eqref{initialreg} and \eqref{initialsep} 
 be satisfied. Then there is a constant $K_2>0$,
which depends on the data of the system, but neither on $u\in\uad$ nor on 
$\alpha\in (0,1]$, such that the following holds true: whenever $u\in\uad$ and 
$\alpha\in (0,1]$ are given and $(\mual,\phial,\sial)=\S_\alpha(u)$, then it holds the estimate
\begin{align}
\label{boundal} 
&\alpha\|\mual\|_{W^{2,\infty}(0,T;V^*)} +\sqrt{\alpha}\|\mual\|_{W^{1,\infty}(0,T;H)}
+\|\mual\|_{L^\infty(0,T;V)}\nonumber\\
&+\|\phial\|_{W^{1,\infty}(0,T;H)\cap H^1(0,T;V)\cap L^\infty(0,T;W)\cap C^0(\overline Q)}
\nonumber\\
&+\|\sial\|_{H^1(0,T;H)\cap C^0([0,T];V)\cap L^2(0,T;W)} 
+\|F_1'(\phial)\|_{L^\infty(0,T;H)}\,\le\,K_2\,.
\end{align}
Moreover, if  $\{\alpha_n\}\subset (0,1]$ and $\{u_{\alpha_n}\}\subset\uad$
are sequences such that $\alpha_n\searrow0$ and $u_{\alpha_n}\to \bar u$ weakly star
in $\Liq$ for some $\bar u\in\Liq$, then the following holds true:
the state variables $(\mu_{\alpha_n},\ph_{\alpha_n},\sigma_{\alpha_n})
:=\S_{\alpha_n}(u_{\alpha_n})$ enjoy for $n\to\infty$ the following convergence properties:
\begin{align}
\label{conmu}
&\mu_{\alpha_n} \to\mu_0 \quad \,\mbox{weakly star in }\,\L\infty V,\\
\label{conmut}
&\alpha_n\mu_{\alpha_n} \to 0 \,\,\,\, \mbox{weakly star in }\,\W{2,\infty}{V^*} \,  
\mbox{ and strongly in }\,\W{1,\infty}H,\\
\label{conphi}
&\ph_{\alpha_n}\to\ph_0 \quad \mbox{weakly star in } \,W^{1,\infty}(0,T;H)\cap H^1(0,T;V)\cap L^\infty(0,T;W)
\nonumber
\\
&\hspace*{20mm} \,
\mbox{ and strongly in }\,\C0V \cap C^0(\overline Q),\\
\label{conxi}
&F_1'(\vp_{\alpha_n})\to\xi_0 \quad \mbox{weakly star in } L^\infty(0,T;H),
\\
\label{consig}
&\sigma_{\alpha_n}\to\sigma_0 \,\quad \mbox{weakly star in } 
\H1 H \cap \L\infty V \cap \L2 {W},
\end{align}
where $(\mu_0,\ph_0,\sigma_0,\xi_0)=\S_0(\bar u)$.
\Ethm
\Bdim
The bound \eqref{boundal} has been proved in the proof of \cite[Thm.~4.1]{CRS1}
for fixed $u\in\uad$ and $\alpha\in (0,1]$. A closer inspection of the estimates performed
there reveals that the derived bounds are in fact uniform with respect to the choice
of $u\in\uad$ and $\alpha\in (0,1]$.  

Now suppose that sequences $\{\alpha_n\}\subset (0,1]$ and $\{u_{\alpha_n}\}
\subset\uad$ are given such that $\alpha_n\searrow0$ and $u_{\alpha_n}\to \bar u$ 
weakly star in $\Liq$. Then, clearly, $\bar u\in\uad$. Moreover, the state variables 
$(\mu_{\alpha_n},\ph_{\alpha_n},\sigma_{\alpha_n})=\S_{\alpha_n}(u_{\alpha_n})$ satisfy
the global bound \eqref{boundal}. Therefore we may without loss of generality assume
that there exists some quadruple $(\bar \mu,\bar \ph,\bar\sigma,\bar\xi)$ such that
the convergence properties \eqref{conmu}--\eqref{consig} are fulfilled \pier{for a subsequence of $\{\alpha_n\}$} with 
$(\mu_0,\ph_0,\sigma_0,\xi_0)$ replaced by $(\bar\mu,\bar\ph,\bar\sigma,\bar\xi)$.
It remains to show that $(\bar\mu,\bar\ph,\bar\sigma,\bar\xi)=\S_0(\bar u)$.
To this end, we first infer from \eqref{conphi}, using \ref{F:weak} and \ref{h:weak}, that
\Beq
F_2'(\ph_{\alpha_n}) \to F_2'(\bar\ph) \,\,\,\mbox{and}\,\,\, \h(\ph_{\alpha_n})\to \h(\bar\ph)\,\,\,\mbox{\pier{pointwise,} \ as}\,\,\, n\to\infty. \nonumber
\Eeq
Moreover, combining \eqref{conphi} and \eqref{conxi} with the maximal monotonicity
of the subdifferential  graph $\partial F_1$ in $L^2(0,T;H)$, we deduce from \cite[Prop.~2.2, 
p.~38]{Barbu} that $\,\bar\xi\in\partial F_1(\bar\ph)\,$ a.e. in $Q$. Hence, we may
pass to the limit as $n\to\infty$ in the system \eqref{var:1}--\eqref{var:4},
written for $\alpha=\alpha_n$ and the control $u=u_{\alpha_n}$, to arrive at
the conclusion that the quadruple $(\bar\mu,\bar\ph,\bar\sigma,\bar\xi)$ is a solution 
to the 
system \eqref{var:1}--\eqref{var:4} with $\alpha=0$ and the control
$u=\bar u$ (where the first two initial conditions in \eqref{var:4} are not 
present).  Therefore, \pier{it turns out that} $(\bar\mu,\bar\ph,\bar\sigma,\bar\xi)
=\S_0(\bar u)$, which concludes the proof of the assertion.
\Edim

\section{Existence and approximation of optimal controls}
\setcounter{equation}{0}
\label{OPT}
Beginning with this section, we study for $\alpha\in [0,1]$ the optimal control problem  {\bf (CP$_\alpha$)} of minimizing the cost functional \eqref{cost} subject to the control
constraint $u\in\uad$ and  the state 
system \Statesys, written in the form \eqref{var:1}--\eqref{var:4}, where in the case 
$\alpha=0$ the first two initial conditions in \eqref{var:4} are not present. In addition to
\ref{const:weak}--\ref{h:weak}, we impose the following general assumptions:
\pier{
\begin{enumerate}[label={\bf (A\arabic{*})}, ref={\bf (A\arabic{*})}, start=5]
\item \label{cbox}
	The conditions \eqref{uad} and \eqref{thresholds} are fulfilled.
\item \label{coeff}
     $b_1\ge 0$, $b_2\ge 0$, $b_3>0$, and $\kappa\ge 0$. 
\item \label{targets}
     $\widehat\ph_Q\in\Lzq$,  $\widehat\ph_\Omega \in V $, \pier{and $\G:L^2(Q)\to [0,+\infty) $ is a convex and continuous functional.}
\end{enumerate}}%
\vspace*{2mm}\noindent
We expect that the minimizers of the control problem {\bf (CP$_\alpha$)} are 
for $\alpha\searrow0$ closely related to those of {\bf (CP$_0$)}.
We have the following result.          
\Bprop
Suppose that \ref{const:weak}--\ref{targets}, \eqref{initialreg} and \eqref{initialsep} 
 are fulfilled, and let sequences $\{\alpha_n\}\subset
(0,1]$ and $\{u_{\alpha_n}\}\subset\Uad$ be given such that $\alpha_n\searrow0$ and $u_{\alpha_n}\to u$ weakly star in $\Liq$ for some $u\in\Uad$.
Then we have
\begin{align}
\label{uwe1}
\J(\S_0^2(u),u)\,&\le\,\liminf_{n\to\infty}\, \J({\cal S}_{\alpha_n}^2
(u_{\alpha_n}),u_{\alpha_n}),\\
\label{uwe2}
\J(\S_0^2(v),v)\,&=\,\lim_{n\to\infty} \, \J({\cal S}_{\alpha_n}^2(v),v) 
\quad\forall\,v\in\Uad.
\end{align}
\Eprop 

\begin{proof}  
According to Theorem 2.3, we have 
$\S_{\alpha_n}^i(u_{\alpha_n}) \to \S_0^i(u)$, for $i=1,2,3$,  in the sense of \eqref{conmu},\eqref{conphi} and \eqref{consig}, respectively. In particular,
$$
\ph_{\alpha_n}:=\S^2_{\alpha_n}(u_{\alpha_n})\to \S_0^2(u) =:\ph_0\quad\mbox{strongly in }\,C^0(\overline Q)\,.
$$ 
Now observe that  the third and fourth summands of the cost functional \eqref{cost} are weakly sequentially lower semicontinuous on $\Lzq$. 
We therefore conclude that
\begin{align}
\label{typisch}
&\liminf_{n\to\infty}\,\J(\S^2_{\alpha_n}(u_{\alpha_n}),u_{\alpha_n})\nonumber\\
&=\,\liminf_{n\to\infty} \Big(\frac{b_1}2\int_Q|\ph_{\alpha_n}-\widehat \ph_Q|^2
+\frac{b_2}2\iO |\ph_{\alpha_n}(T)-\widehat\ph_\Omega |^2+\frac{b_3}2\int_Q|
u_{\alpha_n}|^2+\kappa\G(u_{\alpha_n})\Big)\nonumber\\
&\ge\,\frac{b_1}2 \int_Q|\ph_0-\widehat\ph_Q|^2+\frac{b_2}2 \iO|\ph_0(T)-\widehat\ph
_\Omega|^2+\frac{b_3}2\int_Q|u|^2+\kappa\G(u)\nonumber\\
&=\,\J(\S_0^2(u),u)\,,
\end{align}
and the validity of \eqref{uwe1} follows.

Now let $v\in\Uad$ be 
arbitrarily chosen. Applying Theorem 2.3 to the constant sequence $v_n=v$, $n\in\enne$, we infer that 
$$\ph_{\alpha_n}:=\S_{\alpha_n}^2(v)\to \S_0^2(v)=:\ph_0 
\quad\mbox{strongly in }\,C^0(\overline Q).$$ 
From this and the fact that $v\in\uad$ is fixed, the validity of \eqref{uwe2} readily follows.
\end{proof}

Next, we discuss the existence of minimizers for the control problem {\bf (CP$_\alpha$)}. 
For $\alpha>0$ and the regular case $(r_-,r_+)=\erre$, this has been shown in \cite[Theorem 3.3]{CRS2}. Exactly the same proof can be employed under the more general assumption  \ref{F:weak}, since the uniform separation condition \eqref{separaal}
is satisfied for fixed $\alpha>0$.  For $\alpha=0$, we have the following result.
\Bcor
Suppose that \ref{const:weak}--\ref{targets}, \eqref{initialreg} and \eqref{initialsep} are fulfilled. Then the problem {\bf (CP$_0$)}
admits at least one solution in $\Uad$.
\Ecor
\begin{proof}
We pick an arbitrary sequence $\{\alpha_n\}\subset (0,1]$ such that $\alpha_n\searrow0$ as $n\to\infty$. As stated above, the problem {\bf (CP$_{\alpha_n}$)} has for every $n\in\enne$
a solution 
$\,u_{\alpha_n}\in\Uad\,$ with associated state 
$(\ph_{\alpha_n},\mu_{\alpha_n},\sigma_{\alpha_n})=\S_{\alpha_n}(u_{\alpha_n})$.
 Since $\Uad$ is bounded in $\Liq$, we may without loss
of generality assume that $u_{\alpha_n}\to \bar u$ weakly star in $\Liq$ for some 
$\bar u\in\Uad$. 

Now, we  infer from Theorem 2.3 that $(\mu_{\alpha_n},\ph_{\alpha_n},
\sigma_{\alpha_n}, F_1'(\ph_{\alpha_n})\pier{)}$  converges in the sense of
\eqref{conmu}--\eqref{consig} to $(\mu_0,\ph_0,\sigma_0,\xi_0):=\S_0(\bar u)$. 
Using Proposition 3.1, we then obtain from the optimality of 
$((\mu_{\alpha_n},\ph_{\alpha_n},\sigma_{\alpha_n}),u_{\alpha_n})$ for {\bf (CP$_{\alpha_n}$)}, 
the following chain of (in)equalities for every $v\in\uad$:
\begin{align*}
&\J(\S_0^2(\bar u),\bar u)\,\le\,\liminf_{n\to\infty} \,\J({\cal S}_{\alpha_n}^2(u_{\alpha_n}),u_{\alpha_n})\,\le \, 	\liminf_{n\to\infty}\,\J({\cal S}^2_{\alpha_n}(v),v) \,=\,\J(\S_0^2 (v),v)\,.
\end{align*}
This shows that \betti{$(\S_0(\bar u),\bar u)$} is an optimal pair of the control problem {\bf (CP$_0$)},
 which concludes the proof 
of the assertion.
\end{proof}

The proof of Corollary~3.2 suggests that optimal controls of {\bf (CP$_\alpha$)}
are ``close'' to optimal
controls of {\bf (CP$_0$)} as $\alpha$ approaches zero. However, it does not yield any information on 
whether every optimal control of {\bf (CP$_0$)} can be approximated in this way. In fact, such a global result cannot be expected to hold true.
Nevertheless, a local answer can be given by employing a well-known trick. To this end, 
let $u^*\in\Uad$ be an optimal control for {\bf (CP$_0$)} with associated state
$(\mu^*,\ph^*,\sigma^*,\xi^*)=  \S_0(u^*)$. We  associate
with this optimal control the {\em adapted cost functional}
\begin{equation}
\label{adcost}
\widetilde{\cal J}(\vp,u):=
{\cal J}(\vp,u)\,+\,\frac 12\,\|u-u^*\|^2_{L^2(Q)}
\end{equation}
and a corresponding \emph{adapted optimal control problem} for $\alpha>0$, namely:

\vspace{2mm}\noindent
${\bf (ACP_\alpha)}$ \quad Minimize $\,\, 
\widetilde {\cal J}(\vp,u)\,\,$
for $\,u\in\Uad$\,  subject to the state system \eqref{var:1}--\eqref{var:4}.

\vspace{2mm}\noindent
A standard argument (see the proof of \cite[Thm.~3.3]{CRS2}), which needs no repetition here, \pier{shows} that the 
adapted optimal control problem 
{\bf (ACP$_\alpha$)} has for every $\alpha>0$ at least one solution.
The following result gives a partial answer to the question raised above
concerning the approximation of optimal controls for {\bf (CP$_0$)} by the approximating problems {\bf (ACP$_\alpha$)}.

\begin{theorem} 
Suppose that \ref{const:weak}--{\bf(A7)}, \eqref{initialreg} and \eqref{initialsep} are fulfilled, assume that
$u^*\in \Uad$ is an arbitrary optimal control of {\bf (CP$_0$)} with associated state  
$(\mu^*,\ph^*,\sigma^*,\xi^*)=\S_0(\us)$, and let $\,\{\alpha_n\}_{n\in\enne}\subset 
(0,1]\,$ be
any sequence such that $\,\alpha_n\searrow 0\,$ as $\,n\to\infty$. 
Then, for any $n\in \enne$, there exists 
an optimal control
 $\,u_{\alpha_n}\in \Uad\,$ of the adapted problem 
{\bf (ACP$_{\alpha_n}$)}
 with associated state $(\mu_{\alpha_n},\vp_{\alpha_n},\sigma_{\alpha_n})=\S_{\alpha_n}
 (u_{\alpha_n})$
 such that, as $n\to\infty$,
\begin{align}
\label{ustrong}
&u_{\alpha_n}\to u^* \quad\mbox{strongly in }\,L^2(Q),
\end{align}
and such that \eqref{conmu}--\eqref{consig} hold true with 
$(\mu_0,\vp_0,\sigma_0,\xi_0)$ replaced by $(\mus,\phis,\sis,\xis)$. 
 Moreover, we have 
\begin{align}
\label{convJad}
&\lim_{n\to\infty}\,\widetilde{{\cal J}}(\S^2_{\alpha_{n}}(u_{\alpha_{n}}),u_{\alpha_{n}})
\,=\,{\cal J}(\S_0^2(u^*),u^*).
\end{align}
\end{theorem}
\begin{proof}
For any $ k\in\enne$, we pick an optimal control
$u_{\alpha_k} \in \Uad\,$ for the adapted problem {\bf (ACP$_{\alpha_k}$)} and denote  
$(\mu_{\alpha_k},\vp_{\alpha_k},\sigma_{\alpha_k})=\S_{\alpha_k}(u_{\alpha_k})$ for
$k\in\enne$.
Then, by the boundedness of $\Uad$ in $\Liq$, there is some subsequence $\{\alpha_{n}\}$ 
of $\{\alpha_k\}$ such that
\begin{equation}
\label{ugam}
u_{\alpha_{n}}\to \bar u\quad\mbox{weakly star in }\,\Liq
\quad\mbox{as }\,n\to\infty,
\end{equation}
for some $\bar u\in\Uad$. 
Thanks to Theorem~2.3, the convergence properties \eqref{conmu}--\eqref{consig} hold true correspondingly, where $(\mu_0,\vp_0,\sigma_0,\xi_0)=\S_0(\bar u)$, 
and  the pair $(\S_0^2(\bar u),\bar u)=(\betti{\vp_0},\bar u)$
is admissible for {\bf (CP$_0$)}. 

We now aim at showing that $ \bar u=u^*$. 
Once this is shown, it follows from the unique solvability of the state system 
that also $(\mu_0,\vp_0,\sigma_0,\xi_0)=(\mus, \ph^*,\sis,\xis)$. 

Now observe that, owing to the weak sequential lower semicontinuity properties of 
$\widetilde {\cal J}$, 
and in view of the optimality property of $(\S_0^2(u^*),u^*)$ 
for problem {\bf (CP$_0$)},
\begin{align}
\label{tr3.6}
\liminf_{n\to\infty}\, \widetilde{\cal J}(\S^2_{\alpha_n}(u_{\alpha_n}),u_{\alpha_n})\,
&\ge \,{\cal J}(\S^2_0(\bar u),\bar u)\,+\,\frac{1}{2}\,
\|\bar u-u^*\|^2_{L^2(Q)}\nonumber\\[1mm]
&\geq \, {\cal J}(\S^2_0(u^*),u^*)\,+\,\frac{1}{2}\,\|\bar u-u^*\|^2_{L^2(Q)}\,.
\end{align}
On the other hand, the optimality property of  
$\,(\S^2_{\alpha_{n}}(u_{\alpha_n}),u_{\alpha_n})
\,$ for problem {\bf (ACP$_{\alpha_n}$)} yields that
for any $n\in\enne$ we have
\begin{equation}
\label{tr3.7}
\widetilde {\cal J}({\cal S}^2_{\alpha_{n}}(u_{\alpha_{n}}),
u_{\alpha_{n}})\,\le\,\widetilde {\cal J}({\cal S}^2_{\alpha_n}
(u^*),u^*)\,=\, {\cal J}({\cal S}^2_{\alpha_n}
(u^*),u^*)\,,
\end{equation}
whence, taking the limit superior as $n\to\infty$ on both sides and invoking (\ref{uwe2}) in
Proposition~3.1,
\begin{align}
\label{tr3.8}
	\nonumber
&\limsup_{n\to\infty}\,\widetilde {\cal J}(\S^2_{\alpha_n}(u_{\alpha_n}),
u_{\alpha_n})\,\le\,\limsup_{n\to\infty}\widetilde {\cal J}(\S^2_{\alpha_n}(u^*),u^*) 
\\ 
& =\,\limsup_{n\to\infty} {\cal J}(\S^2_{\alpha_n}(u^*),u^*)
\,=\,{\cal J}(\S_0^2(u^*),u^*)\,.
\end{align}
Combining (\ref{tr3.6}) with (\ref{tr3.8}), we have thus shown that 
$\,\frac{1}{2}\,\|\bar u-u^*\|^2_{L^2(Q)}=0$\,,
so that $\,\bar u=u^*\,$  and thus also $(\mu_0,\ph_0,\sigma_0,\xi_0)=
(\mus,\ph^*,\sis,\xis)$. 
Moreover, (\ref{tr3.6}) and (\ref{tr3.8}) also imply that
\begin{align*}
&{\cal J}(\S_0^2(u^*),u^*) \, =\,\widetilde{\cal J}(\S_0^2(u^*),u^*)
\,=\,\liminf_{n\to\infty}\, \widetilde{\cal J}(\S^2_{\alpha_n}(u_{\alpha_n}),
 u_{\alpha_{n}})\nonumber\\[1mm]
&\,=\,\limsup_{n\to\infty}\,\widetilde{\cal J}(\S^2_{\alpha_n}(u_{\alpha_n}),
{u_{\alpha_n}}) \,
=\,\lim_{n\to\infty}\, \widetilde{\cal J}(\S^2_{\alpha_n}(u_{\alpha_n}),
u_{\alpha_n})\,,
\end{align*}                                    
which proves the validity of {(\ref{convJad})}.
Moreover, the convergence properties 
\eqref{conmu}--\eqref{consig} are satisfied with $(\mu_0,\ph_0,\sigma_0,\xi_0)$ 
replaced by $(\mus,\ph^*,\sis,\xis)$. On the other hand, we have that 
\begin{align*}
{\cal J}(\S_0^2(u^*),u^*) 
\,&\leq\,\liminf_{n\to\infty}\, {\cal J}(\S^2_{\alpha_n}(u_{\alpha_n}),
 u_{\alpha_{n}})
 \,\leq\,\limsup_{n\to\infty}\, {\cal J}(\S^2_{\alpha_n}(u_{\alpha_n}),
 u_{\alpha_{n}}) \nonumber\\[1mm]
&\leq \,\limsup_{n\to\infty}\,\widetilde{\cal J}(\S^2_{\alpha_n}(u_{\alpha_n}),
{u_{\alpha_n}}) \,
=\,{\cal J}(\S_0^2(u^*),u^*) ,
\end{align*}        
so that also 
$ {\cal J}(\S^2_{\alpha_n}(u_{\alpha_n}),{u_{\alpha_n}}) \to  {\cal J}(\S_0^2(u^*),u^*)$ as 
$n\to \infty $, and the relation in \eqref{adcost} enables us to infer the strong convergence in \eqref{ustrong}
for the subsequence $\{u_{\alpha_n}\}$. 

We now claim that \eqref{ustrong} and \eqref{convJad} hold true even for the entire sequence, due to the complete identification of the limit
$\bar u$ as $u^*$. We only prove this claim for \eqref{ustrong}; for \eqref{convJad} a similar reasoning may be used.
Assume that \eqref{ustrong} were not true. Then there exist some $\varepsilon>0$ and a subsequence $\{\alpha_j\}$ of  $\{\alpha_k\}$ such that  
\begin{equation}
\label{sirtoby}
\|u_{\alpha_j}-u^*\|_{L^2(Q)}\,\ge\,\varepsilon \quad\forall\,j\in\enne.
\end{equation}
However, by the boundedness of $\Uad$, there is some subsequence $\{\alpha_{j_n}\}$ 
of $\{\alpha_j\}$ such that, with some $\tilde u\in\Uad$,
\begin{equation*}
u_{\alpha_{j_n}}\to \tilde u\quad\mbox{weakly star in }\,\Liq
\quad\mbox{as }\,n\to\infty\,.
\end{equation*}
Arguing as above, \pier{we deduce that $\tilde{u} = u^*$ and that \eqref{ustrong} also holds for the subsequence $\{u_{\alpha_{j_n}}\}$. This yields a contradiction, since \eqref{sirtoby} clearly implies that the sequence $\{u_{\alpha_j}\}$ cannot admit any subsequence converging} strongly to $\,u^*\,$ in $L^2(Q)$.
\end{proof}

\section{Convergence of the adjoint state variables}
\setcounter{equation}{0}
\label{ADJ} 
Assume that the general \betti{hypotheses}  \ref{const:weak}--\ref{targets}, \eqref{initialreg} and
\eqref{initialsep} are fulfilled. We study in this section the adjoint \betti{problem} associated with the state system~\Statesys.
To this end, let $\alpha\in(0,1]$ be arbitrary and $u\in\uad$ be any admissible control with associated state $(\mual,\phial,\sial)=\S_\alpha(u)$. Then, according to 
\cite[Sect.~3~and~Sect.~4, Scenario 1]{CRS2}, the adjoint state system for the optimal control problem {\bf (CP$_\alpha$)} with the cost functional \eqref{cost} and the adjoint variables $(p,q,r)$ \betti{has} 
the form
\begin{align}
\label{as1}
&\alpha\dtt p-\Delta p = q-P_0(p-r)  \quad\mbox{a.e. in }\,Q,
\\[1mm]
\label{as2}
&-\dt(\tau q +p)-\Delta q=-\chi\Delta r-F''(\phial)q-\chi P_0(p-r)
-\h'(\phial)\,z\,p
\nonumber\\
& \hspace*{39mm}+ b_1(\phial-\widehat\ph_Q) \quad\mbox{a.e. in }\,Q,\\[1mm]
\label{as3}
&-\dt r-\Delta r=\chi q+P_0(p-r)\quad\mbox{a.e. in }\,Q,\\[1mm]
\label{as4}
&\dn p=\dn q=\dn r=0 \quad\mbox{a.e. on }\,\Sigma,
\\
\label{as5}
&p(T)=0, \,\,\, \dt p(T)=0, \,\,\, (\tau q+p)(T)=b_2(\phial(T)-\pier{\widehat\ph_\Omega}), \,\,\,
r(T)=0\quad\mbox{a.e. in }\,\Omega.
\end{align}
We remark already at this point that  \eqref{as1}--\eqref{as5} is also the adjoint
system for the  adapted optimal control problem {\bf (ACP$_\alpha$)} in which the
cost functional  is given by \eqref{adcost} with some $\us\in\uad$.

Now, by virtue of \cite[Thm.~3.7]{CRS2}, the system \eqref{as1}--\eqref{as5}  enjoys 
for every $\alpha\in (0,1]$ a unique solution $(\pal,\qal,\ral)$ with the
regularity
\begin{align}
\label{regpal}
&\pal\in W^{2,\infty}(0,T;H)\cap W^{1,\infty}(0,T;V)\cap L^\infty(0,T;W)\cap 
C^0(\overline Q),\\
\label{regqal}
&\qal \in H^1(0,T;H)\cap C^0([0,T];V) \cap L^2(0,T;W),\\
\label{regral}
&\ral\in  H^1(0,T;H)\cap C^0([0,T];V)\cap L^2(0,T;W)\cap \Liq.
\end{align}

Next, we derive estimates for the adjoint variables $(\pal,\qal,\ral)$ that are \pier{independent of} $\alpha\in (0,1]$. We have the following result.
\Blem
Suppose that \ref{const:weak}--\ref{targets}, \eqref{initialreg} and \eqref{initialsep} are satisfied. Then
there exists a constant $K_3>0$, which  depends on the data of the system, but not on
$\alpha\in (0,1]$, such that the following holds true: whenever $u\in\uad$ 
and $\alpha\in (0,1]$ are given and
$(\pal,\qal,\ral)$ is the unique solution to the system \eqref{as1}--\eqref{as5} 
with $(\mual,\phial,\sial)=\S_{\alpha}(u)$, then 
\begin{align}
\label{boundadj}
&\sqrt{\alpha}\|\pal\|_{W^{1,\infty}(0,T;\pier{H})}
+\|\pal\|_{L^\infty(0,T;V)}+\|\qal\|_{L^\infty(0,T;H)\cap L^2(0,T;V)} 
+\int_Q F_1''(\phial)|\qal|^2
\nonumber\\
&+\|\ral\|_{H^1(0,T;H)\cap C^0([0,T];V)\cap L^2(0,T;W)\cap\Liq} \,\le\,K_3.
\end{align} 
\Elem
\Bdim
In the following, we denote by $C$ any positive constant that may depend on the data of
the system but not on the choice of either $u\in\uad$ or $\alpha\in (0,1]$. We also introduce
the quantities
\Beq
\label{gal}
g_1^\alpha(u):=b_1(\phial-\widehat\ph_Q)\in L^2(0,T;H), \quad g_2^\alpha(u):=b_2(\phial(T)-
\widehat\ph_\Omega)\in V.
\Eeq
 We first note that the state variable $\phial$ satisfies the global bound \eqref{ssbound1},
 and thus
\Beq
\label{boundalg}
\|g_1^\alpha(u)\|_{L^2(0,T;H)} + \|g_2^\alpha(u)\|_V\,\le\,C.
\Eeq 
In addition, we can infer from \ref{F:weak} and \ref{h:weak} that
\Beq
\label{boundalh}
\|F_2''(\phial)\|_\Liq +\|\h'(\phial)\|_\Liq\,\le\,C.
\Eeq
Now we consider the identities \eqref{as1}--\eqref{as5} for $(\pal,\qal,\ral)$.
We multiply \eqref{as1} by $\,-\dt\pal$, \eqref{as2} by \,$\qal$, and \eqref{as3}
by $\,-K\dt\ral$  (where $K>0$ is yet to be chosen), add the three resulting equations
and integrate over $\,Q_t^T:=\Omega \times (t,T)$, where $t\in [0,T)$ is arbitrary.
Then, using the global bounds \eqref{boundalg} and \eqref{boundalh}, as well as \pier{\ref{P:weak} and} Young's
inequality, we find that
\begin{align}
\label{Omiabella}
&\frac \alpha 2\|\dt\pal(t)\|^2+\frac 12\iO|\nabla \pal(t)|^2 +\frac {P_0}2\|\pal(t)\|^2
+\frac \tau 2\|\qal(t)\|^2 +\frac K2\iO|\nabla \ral(t)|^2 \nonumber\\
&+\frac {K P_0}2\|\ral(t)\|^2 +\Qtt |\nabla\qal|^2 + K\Qtt|\dt\ral|^2
+\Qtt F_1''(\phial)|\qal|^2\nonumber\\
&=\,\frac {b_2^2}{2\tau}\|\phial(T)-\widehat\ph_\Omega\|^2-P_0\Qtt \ral\dt \pal
+\chi\Qtt \nabla\qal\cdot\nabla\ral -\Qtt F_2''(\phial)|\qal|^2\
\nonumber\\
&\quad \,\,-\chi P_0\Qtt(\pal-\ral)\qal -\Qtt \h'(\phial)z\,\pal\qal
+b_1\Qtt (\phial-\widehat\ph_Q)\qal
\nonumber\\
&\quad \,\, -K\chi\Qtt \qal\dt\ral-K\pier{P_0}\Qtt \pal\dt\ral\nonumber\\
&\le\, C+C_K\Qtt\bigl(|\pal|^2+|\qal|^2+|\ral|^2\bigr)+\frac 12\Qtt|\nabla\qal|^2
+C\Qtt|\nabla\ral|^2
\nonumber\\
&\quad\,\,+\frac K4\Qtt|\dt\ral|^2-P_0\Qtt\ral\dt\pal\,.
\end{align}
Observe that the last term on the \lhs\ is nonnegative, since $F_1$ is convex.
Now, in view of the endpoint conditions \eqref{as5} and Young's inequality,
we have that
\begin{align}
\label{Napoli}
&-P_0\Qtt \ral\dt\pal\,=\,P_0\iO\ral(t)\pal(t)+P_0\Qtt \pal\dt\ral \nonumber\\
&\le\,\frac{P_0}4\|\pal(t)\|^2+P_0\|\ral(t)\|^2+\frac K4\Qtt|\dt\ral|^2
+C_K\Qtt|\pal|^2\,.
\end{align}
At this point, we choose $K=4$, \pier{use \eqref{Napoli} in the \rhs\ of \eqref{Omiabella},}
and rearrange the terms. It then turns out that Gronwall's lemma can be applied backward
in time, whence we conclude that
\begin{align}
&\sqrt{\alpha}\|\dt\pal\|_{L^\infty(0,T;H)} + \|\pal\|_{L^\infty(0,T;V)}
+\|\qal\|_{L^\infty(0,T;H)\cap L^2(0,T;V)}\nonumber\\
&+\|\ral\|_{H^1(0,T;H)\cap L^\infty(0,T;V)}+\int_Q F_1''(\phial)|\qal|^2\,\le\,C\,.
\nonumber
\end{align} 
Comparison in \eqref{as3} and standard \pier{elliptic} estimates then show that also
$$
\|\ral\|_{L^2(0,T;W)}\,\le\,C, 
$$
whence, by the continuity of the embedding of $H^1(0,T;H)\cap L^2(0,T;W)$ in
$C^0([0,T];V)$, also
\Beq
\|\ral\|_{C^0([0,T];V)}\,\le\,C \nonumber
\Eeq
follows. Finally, we notice that the \rhs\ of \eqref{as3} is already shown to be 
uniformly bounded in $L^\infty
(0,T;H)$. Moreover, $\ral$ satisfies zero terminal and Neumann boundary conditions.
It then follows from standard parabolic theory (cf. \cite[Thm.~7.1, ~p.181]{LSU}) that 
$$
\|\ral\|_\Liq \,\le\,C,
$$ 
which concludes the proof of the assertion.
\Edim

Next, we prove an approximation result for the adjoint state variables.
\Bthm
Suppose that \ref{const:weak}--\ref{targets}, \eqref{initialreg} and \eqref{initialsep} are satisfied, and 
let $\{\alpha_n\} \subset (0,1]$ and $\{u_{\alpha_n}\} \subset \uad$ be sequences
 such that $\alpha_n\searrow0$ and $u_{\alpha_n}\to u_0$ weakly star in $\Liq$ 
 for some $u_0\in\uad$. 
Then there exist a subsequence $\{\alpha_{n_k}\}$ of $\{\alpha_n\}$ and a triple
$(p_0,q_0,r_0)$ such that, as $k\to\infty$,
\begin{align}
\label{Conny1}
&\alpha_{n_k} p_{\alpha_{n_k}} \to 0  \quad\mbox{weakly star in \,$W^{2,\infty}(0,T;H)$\,and strongly in }\,W^{1,\infty}(0,T;V),\\  
\label{Conny2}
&p_{\alpha_{n_k}}\to p_0  \quad\mbox{weakly star in }\,L^\infty(0,T;W), \\
\label{Conny3}
&q_{\alpha_{n_k}} \to q_0  \quad\mbox{weakly star in }\,L^\infty(0,T;V)\cap L^2(0,T;W),\\
&\tau q_{\alpha_{n_k}}+p_{\alpha_{n_k}}\to \tau q_0+p_0 \quad\mbox{weakly in }\,H^1(0,T;H)
\nonumber
\\
&\hspace*{22.6mm} \pier{\mbox{and strongly in }\,C^0([0,T];H)\cap L^2(0,T;V),}
\label{Conny4}
\\
&r_{\alpha_{n_k}} \to r_0 \,\quad\mbox{weakly star in }\, H^1(0,T;H)\cap L^\infty([0,T];V)
\cap
L^2(0,T;W) \cap\Liq\nonumber\\
&\hspace*{22.6mm} \mbox{and strongly in }\,C^0([0,T];H)\cap L^2(0,T;V).
\label{Conny5}
\end{align}
Moreover,  with $\ph_0=\S_0^2(u_0)$, $(p_0,q_0,r_0)$ is a solution to the system
\begin{align}
\label{aso1}
&-\Delta p = q-P_0(p-r)  \quad\mbox{a.e. in }\,Q,
\\[1mm]
\label{aso2}
&-\dt(\tau q +p)-\Delta q=-\chi\Delta r-F''(\ph_0)q-\chi P_0(p-r)
-\h'(\ph_0)\,z\,p
\nonumber\\
& \hspace*{39mm}+ b_1(\ph_0-\widehat\ph_Q) \quad\mbox{a.e. in }\,Q,\\[1mm]
\label{aso3}
&-\dt r-\Delta r=\chi q+P_0(p-r)\quad\mbox{a.e. in }\,Q,\\[1mm]
\label{aso4}
&\,\dn p=\dn q=\dn r=0 \quad\mbox{a.e. on }\,\Sigma,
\\
\label{aso5}
& \,(\tau q+p)(T)=b_2(\ph_0(T)-\ph_\Omega), \,\,\,
r(T)=0,\quad\mbox{a.e. in }\,\Omega.
\end{align}

\Ethm

\begin{proof}
Let $\{\alpha_n\}\subset (0,1]$ and $\{u_{\alpha_n}\}\subset\uad$ be sequences such that
$\alpha_n\searrow0$ and $u_{\alpha_n}\to u_0$ weakly star in $\Liq$. 
According to Theorem 2.3, then $(\mu_{\alpha_n},\ph_{\alpha_n},\sigma_{\alpha_n})=\S_{\alpha_n}(u_{\alpha_n})$ and
$F_1'(\ph_{\alpha_n})$ converge to $(\mu_0,\ph_0,\sigma_0,\xi_0)=\S_0(u_0)$ in the sense
of \eqref{conmu}--\eqref{consig}, respectively. In particular, $\ph_{\alpha_n}\to \ph_0$ 
strongly in $C^0(\overline Q)$. Hence, by \ref{h:weak} and \ref{F:weak} \pier{we have that}
\Beq
\label{conhF2}
\h'(\ph_{\alpha_n})\to \h'(\ph_0), \quad F_2''(\ph_{\alpha_n})\to F_2''(\ph_0), \quad
\mbox{both strongly in }\,C^0(\overline Q).
\Eeq
At this point, we recall that $\ph_0$ satisfies the separation condition \eqref{separao}. 
Therefore, there exist a compact subinterval $M$ of $(r_-,r_+)$ and some $n_0\in\enne$ such 
that
\Beq
\label{unisep}
\ph_{\alpha_n}(x,t)\in M \quad\mbox{for all \,$(x,t)\in\overline Q$\, and }\,n\ge n_0.
\Eeq 
By \ref{F:weak}, $F_1\in C^3(M)$, so that also
\Beq
\label{conF1}
F_1''(\ph_{\alpha_n})\to F_1''(\ph_0) \quad \mbox{strongly in }\,C^0(\overline Q),
\Eeq
and there is some $C>0$ such that
\Beq
\label{boundF}
\|F_1''(\ph_{\alpha_n})\|_{C^0(\overline Q)} + \|F_2''(\ph_{\alpha_n})\|_{C^0(\overline Q)}\,\le\,C \quad\mbox{for all }\,n\in\enne.
\Eeq
Here, and in the remainder of this proof, we denote by $C$ positive constants that may
depend on the data but not on $n\in \enne$.

Now let, for $n\in\enne$, $(\paln,\qaln,\raln)$ be the adjoint variables associated with $u_{\alpha_n}$, which enjoy the regularity \eqref{regpal}--\eqref{regral} and fulfill the identities \eqref{as1}--\eqref{as5} for $\alpha=\alpha_n$. Then, owing to \pier{Lemma}~4.1,
the global estimate \eqref{boundadj} is satisfied for $\alpha=\alpha_n$, $n\in\enne$.

We are now going to derive stronger estimates, using, in particular, \eqref{boundF}. To this
end, we  test \eqref{as1} by $\,\dt\Delta \paln$ (this is only formal, but can be made 
rigorous on a suitable Faedo--Galerkin approximation) to find that
\begin{align}
\label{bodo1}
&\frac \alpha 2\|\nabla\dt\paln(t)\|^2+\frac12\|\Delta\paln(t)\|^2+\frac{P_0}2
\|\nabla\paln(t)\|^2\nonumber\\
&=\,-\Qtt \nabla\qaln\cdot\nabla\dt\paln +P_0 \Qtt \raln\dt\Delta\paln,  
\end{align}
where, owing to \eqref{boundadj}, the last summand on the \rhs, which we denote by
$I_1$, can in view of \eqref{boundadj} be estimated as follows:
\begin{align}
\label{bodo2}
I_1\,=\,&-P_0\iO \raln(t)\Delta\paln(t)-P_0\Qtt \dt\raln \Delta\paln \nonumber\\
\le\,&C+\frac 14\|\Delta\paln(t)\|^2+ C\Qtt |\Delta\paln|^2\,.
\end{align}

Next, we test \eqref{as2} by $\,-\Delta\qaln$, obtaining the identity
\begin{align}
\label{bodo3}
&\frac \tau 2 \|\nabla\qaln(t)\|^2-\Qtt\nabla\dt\paln\cdot\nabla\qaln +\Qtt |\Delta\qaln|^2
\nonumber\\
&=\frac{b_2^2}{2\tau}\iO |\nabla(\ph_{\alpha_n}(T)-\widehat\ph_\Omega)|^2
+\Qtt Z_{\alpha_n}\Delta\qaln,
\end{align}
where the first term on the \rhs\ is by \eqref{boundalg} known to be bounded and the 
expression 
\begin{align}
Z_{\alpha_n}:=&\chi\Delta\raln+F''(\ph_{\alpha_n})\qaln +\chi P_0(\paln-\raln)+
\h'(\ph_{\alpha_n})z\,\paln-b_1(\ph_{\alpha_n}-\widehat \ph_Q)  \nonumber
\end{align}
is, by virtue of \eqref{boundadj}, \eqref{boundalg}, \eqref{boundalh} and \eqref{boundF},
bounded in $\Lzq$. Hence, adding \eqref{bodo1} and \eqref{bodo3}, whereby two terms
cancel, taking \eqref{bodo2} into account, applying Young's inequality and Gronwall's
lemma backward in time, and using standard elliptic estimates, we finally obtain the estimate
\begin{align}
\label{bodo4}
&\sqrt{\alpha_n}\,\|\paln\|_{W^{1,\infty}(0,T;V)} +\|\paln\|_{L^\infty(0,T;W)}
+\|\qaln\|_{L^\infty(0,T;V)\cap L^2(0,T;W)} \,\le\,C\,.
\end{align}
Comparison in \eqref{as1} and in \eqref{as2}, respectively, then lead to the conclusion
that also
\Beq
\label{bodo5}
\alpha_n \|\paln\|_{W^{2,\infty}(0,T;H)} + \|\tau\qaln+\paln\|_{H^1(0,T;H)}\,\le\,C\,.
\Eeq

From the estimates \eqref{boundadj}, \eqref{bodo4} and \eqref{bodo5} it follows the 
existence of a subsequence $\{\alpha_{n_k}\}$ of $\{\alpha_n\}$ and of a triple
$(p_0,q_0,r_0)$ such that the convergence properties \eqref{Conny1}--\eqref{Conny5}
hold true, where the strong \pier{convergence in~\eqref{Conny4} and~\eqref{Conny5}} 
follows from
well-known compactness results (see, e.g., \cite[Sect.~8, Cor.~4]{Simon}). Moreover, passage to the limit as \juerg{$k\to\infty$} in the equations 
\eqref{as1}--\eqref{as4}, written for $\alpha=\alpha_{n_k}$, readily yields that $(p_0,q_0,r_0)$
is a strong solution to the limiting system \eqref{aso1}--\eqref{aso4}. \pier{Note, in particular, that the terms with a product, i.e. $F''(\ph_{\alpha_{n_k}})q_{\alpha_{n_k}} $ and $
\h'(\ph_{\alpha_{n_k}})z\,p_{\alpha_{n_k}}$, weakly converge due to the strong convergence in \eqref{conhF2}.}
Finally, we obviously
have that $r_0(T)=0$, and we \pier{deduce 
from \eqref{Conny4}}
that \,$(\tau q_0+p_0)(T)=b_2(\ph_0(T)-\widehat\ph_0)$, \pier{as well.}
In conclusion, also \eqref{aso5} is valid. The assertion is thus completely proved.
\end{proof}

\Brem
Notice that we have implicitly shown that the limiting system \eqref{aso1}--\eqref{aso5}
has a solution $(p_0,q_0,r_0)$ with the regularity
\begin{align}
\label{regpo}
&p_0\in L^\infty(0,T;W), \quad q_0\in L^\infty(0,T;V)\cap L^2(0,T;W), \quad
\tau q_0+p_0\in H^1(0,T;H),\nonumber\\
&r_0\in H^1(0,T;H)\cap C^0([0,T];V)\cap L^2(0,T;W)\cap\Liq\,.
\end{align}
We are now going to show that this solution is unique, which then entails that the
convergence properties \eqref{Conny1}--\eqref{Conny5} hold true for the entire 
sequences and not just for suitable subsequences. To this end, let $(p_i,q_i,r_i)$, $i=1,2$,
be two such solutions. Then $(p,q,r):=(p_1-p_2,q_1-q_2,r_1-r_2)$ solves the homogeneous
system \eqref{aso1}--\eqref{aso5}, in which the terms $b_1(\ph_0-\widehat\ph_Q)$ and
$b_2(\ph_0(T)-\widehat\ph_\Omega)$ are replaced by zero.

Let $t\in[0,T)$ be arbitrary but fixed. We multiply \eqref{aso1}, written for $t$,  by $p(t)$\pier{,} integrate over $\Omega$ and apply Young's inequality appropriately to obtain that
\begin{align}
\label{oma1}
&\|\nabla p(t)\|^2 + P_0\|p(t)\|^2=\iO q(t)p(t) + P_0\iO r(t)p(t)
\nonumber\\
& \pier{\juerg{=\,}\frac1\tau \iO ( \tau q+  p) (t) p(t) -\frac1\tau \|p(t)\|^2
+ P_0\iO r(t)p(t)}
\nonumber\\
&\le\,\frac{P_0}2\|p(t)\|^2 \pier{{}+ \frac 1{\tau^2 P_0}\|(\tau q + p)(t)\|^2} + P_0\|r(t)\|^2\,.
\end{align} 
Next, we multiply \eqref{aso2} by $K(\tau q+p)$, where $K>0$ is yet to be chosen,
 and integrate over \juerg{$Q^T_t$}. It follows,
using integration by parts and Young's inequality,
\begin{align}
\label{oma2}
&\pier{\frac{K}2\|(\tau q + p)(t)\|^2}+K\tau \Qtt|\nabla q|^2 \nonumber\\
&=\,-K\Qtt\nabla q\cdot\nabla p + K\chi\Qtt \nabla r\cdot \nabla(\tau q+p)
-K\Qtt F''(\ph_0)q(\tau q+p) \nonumber\\
&\quad\,-K\chi P_0\Qtt (p-r)(\tau q+p) -K\Qtt \h'(\ph_0)z\,p(\tau q +p) \nonumber\\
&\le\,\frac{K\tau}2\Qtt|\nabla q|^2+C_K\Qtt \bigl(|p|^2+\pier{|\tau q +p |^2}+|r|^2+|\nabla p|^2
+|\nabla r|^2\bigr)\,.
\end{align}
Finally, we multiply \eqref{aso3} by $ \,-4\dt r$ and integrate over \betti{$Q^T_t$} and by parts.
Young's inequality then yields that
\begin{align}
\label{oma3}
&4\Qtt |\dt r|^2 + 2\|\nabla r(t)\|^2+2P_0\|r(t)\|^2 \nonumber\\
&=\,- \pier{\frac{4\chi}\tau \Qtt (\tau q +p)} \dt r
-4 \pier{\Bigl(P_0 - \frac{\chi}\tau\Bigr)}\Qtt  p\dt r \nonumber\\
&\le\,\Qtt|\dt r|^2 +C\Qtt\bigl(\pier{|\tau q +p |^2+|p|^2} \bigr)\,.
\end{align}
At this point, we add the \pier{inequalities} \eqref{oma1}--\eqref{oma3} and
rearrange the terms. We also choose 
\juerg{$K>\frac 2{\tau^2 P_0}$ so that $\frac{K} 2>\frac 1{\tau^2P_0}$}. An application of Gronwall's
lemma backward in time then yields that $p=q=r=0$, i.e., we have the asserted uniqueness.
\Erem

\section{Convergence of the first-order necessary optimality conditions}
\setcounter{equation}{0}
\label{CONV}
In this section, we investigate the convergence properties of the first-order necessary 
optimality conditions. Again, we generally assume that the conditions 
\ref{const:weak}--\ref{targets}, \eqref{initialreg} and \eqref{initialsep} are fulfilled. 
Let $\alpha>0$ be given and $u_\alpha\in\uad$ be an optimal control for the control problem {\bf (CP$_\alpha$)} with associated state $(\mual,\phial,\sial)=
\S_{\alpha}(\ual)$ and adjoint state $(\pal,\qal,\ral)$. Then it 
follows from \cite[Thm.~4.1]{CRS2} that the variational
inequality 
\Beq
\label{VUG1}
\int_Q \bigl(\ral+\kappa\lambda_{\alpha} +b_3 u_\alpha\bigr)(u-u_\alpha)\ge 0 \quad
\mbox{for all }\,u\in\uad
\Eeq
is satisfied with some $\lambda_\alpha\in\partial\G(\ual)$. Now assume that $u_0\in\uad$
 is a fixed optimal control for {\bf (CP$_0$)} 
 with associated state $(\mu_0,\ph_0,\sigma_0)=\S_0(u_0)$
 and adjoint state $(p_0,q_0,r_0)$, where the latter is the 
(by Remark 4.3 unique) solution to \eqref{aso1}--\eqref{aso5} having the regularity
properties  \eqref{regpo}. We expect that then, with a suitable 
 $\lambda_0\in \partial \G(u_0)$, it holds the  variational inequality
 \Beq
 \label{VUG2}
 \int_Q \bigl(r_0+\kappa\lambda_0 +b_3 u_0\bigr) (u-u_0)\betti{{}\geq 0} \quad\mbox{for all }\,u\in\uad.
 \Eeq
We are going to show that this is actually the case. Unfortunately, we cannot simply pass
to the limit as $\alpha\searrow0$ in \eqref{VUG1} in order to establish the validity of 
\eqref{VUG2}, since the convergence properties of $\{u_\alpha\}_{\alpha>0}$ are too weak.
Instead, we argue with the optimal control problem {\bf (ACP$_\alpha$)} induced by
the adapted cost functional \eqref{adcost} associated with $\us=u_0$. 
In this connection, we note that it follows
from the analysis in Section 3 of \cite{CRS2} that the 
adjoint state systems (and thus also the uniquely determined adjoint states) 
for the  minimization problems  {\bf (CP$_\alpha$)} and
{\bf (ACP$_\alpha$)} are the same, while the first-order necessary optimality
conditions are different; namely, the variational inequality for the adapted 
problem takes the form 
\Beq
\label{VUG3}
\int_Q \bigl(\ral+\kappa\lambda_{\alpha} +b_3 u_\alpha +\ual-\juerg{u_0}\bigr)(u-u_\alpha)\ge 0, \quad
\mbox{for all }\,u\in\uad,
\Eeq
with some $\lambda_\alpha\in\partial\G(\ual)$ that may differ from the one appearing in
\eqref{VUG1}. We now choose any sequence $\{\alpha_n\} \subset (0,1]$ with $\alpha_n
\searrow0$. Then we can infer from Theorem 3.3  that for any $n\in\enne$ there exists an optimal control $u_{\alpha_n}\in\uad$ of the adapted
problem {\bf (ACP$_{\alpha_n}$)} with associated state $(\mualn,\phialn,\betti{\sigma_{\alpha_n}})
=\S_{\alpha_n}(u_{\alpha_n})$ such that, as $n\to\infty$,
\Beq
\label{ust}
u_{\alpha_n}\to u_0 \quad\mbox{strongly in }\,\Lzq,
\Eeq
and the convergence properties \eqref{conmu}--\eqref{consig} hold 
true. Moreover, Theorem 4.2 and Remark 4.3
imply that the corresponding adjoint states $(\paln,\qaln,\raln)$ 
converge in the sense of \eqref{Conny1}--\eqref{Conny5} to $(p_0,q_0,r_0)$.

At this point, \pier{recalling \ref{targets}} we notice that (see~\cite{Rock}) the subdifferential operator $\partial\G$ is maximal monotone \pier{with domain $ \Lzq$, and therefore continuous and locally bounded in $ \Lzq$.} Now let $\lambda_{\alpha_n}\in\partial\G(u
_{\alpha_n})$ satisfy \eqref{VUG3}, for all $n\in\enne$. Then the sequence $\{\lambda_{\alpha_n}\}$
is bounded in $\Lzq$, and we may without loss of generality assume that, as
$n\to\infty$,
$$
\lambda_{\alpha_n}\to \lambda_0 \quad \mbox{weakly in }\,\Lzq
$$
for some $\lambda_0\in\Lzq$. The demiclosedness of the maximal monotone operator
$\partial\G$ then shows that $\lambda_0\in\partial\G(u_0)$.

At this point, we are ready to pass to the limit as $n\to\infty$ in the variational
inequality \eqref{VUG3}, and the strong convergence \eqref{ust} entails that 
\eqref{VUG2} is in fact satisfied.

\Brem
It follows from a standard argument the validity of the pointwise projection formula \pier{holds:}
\Beq
\label{pointwise}
u_0= \max \left \{\underline{u}, \min \left\{\overline u,-\mbox{$\frac 1{b_3}$} 
(r_0+\kappa\lambda_0)\right\}\right\} \quad\mbox{a.e. in }\,Q.
\Eeq
As detailed in Section 5 of \cite{CRS2} for the case $\alpha>0$, this \pier{formula can be exploited to establish sparsity properties}
for the control problem {\bf (CP$_0$)}.
\Erem

\section*{Acknowledgments}
\pier{PC and ER gratefully mention their affiliation
to the GNAMPA (Gruppo Nazionale per l'Analisi Matematica, 
la Probabilit\`a e le loro Applicazioni) of INdAM (Isti\-tuto
Nazionale di Alta Matematica) and their collaboration
as Research Associate, to the IMATI -- C.N.R. Pavia, Italy.
The project also profited from visits of PC at the Weierstrass Institute in Berlin and of JS at the
Dipartimento di Matematica ``Felice Casorati'' in Pavia. The warm hospitality of these institutions 
is gratefully acknowledged.}


\end{document}